\documentclass{aims2019} 
\usepackage{amsmath} 

\usepackage{graphics} 
\usepackage{epsfig} 
\usepackage{graphicx}  
\usepackage{epstopdf}
\usepackage[colorlinks=true]{hyperref}
\hypersetup{urlcolor=blue, citecolor=red}
\usepackage{hyperref}

\def\currentvolume{X}
 \def\currentissue{X}
  \def\currentyear{200X}
   \def\currentmonth{XX}
	\def\ppages{X--XX}
	 \def\DOI{10.3934/xx.xx.xx.xx}

\newtheorem{theorem}{Theorem}[section]

\newtheorem{lemma}[theorem]{Lemma}

\theoremstyle{definition}

\newtheorem{remark}{Remark}

\newcommand{\ep}{\varepsilon}

\title[Wave equations with localized damping]
{Attractors for semilinear wave equations with localized damping and external forces}

\author[T. F. Ma and P. N. Seminario-Huertas]{}



\email{matofu@icmc.usp.br}

\email{pseminar@icmc.usp.br}

\thanks{$^*$ Corresponding author: T. F. Ma}

\begin{document}
\maketitle

\centerline{\scshape To Fu Ma$^*$ and Paulo N. Seminario Huertas}

\medskip

{
\footnotesize
\centerline{Institute of Mathematical and Computer Sciences, University of S\~ao Paulo}
\centerline{S\~ao Carlos 13566-590, SP, Brazil}
} 

\bigskip

\centerline{\small \it Dedicated to Professor Tom\'as Caraballo on the occasion of his 60th birthday.}


\begin{abstract}
This paper is concerned with long-time dynamics of semilinear wave equations defined on bounded domains of $\mathbb{R}^3$  
with cubic nonlinear terms and locally distributed damping. The existence of regular finite-dimensional global attractors established by Chueshov, Lasiecka and Toundykov (2008) reflects a good deal of the current state of the art on this matter. Our contribution is threefold. First, we prove uniform boundedness of attractors with respect to a forcing parameter. Then, we study the continuity of attractors with respect to the parameter in a residual dense set. 
Finally, we show the existence of generalized exponential attractors. These aspects were not previously considered for wave equations with localized damping. 
\end{abstract}

\section{Introduction} 
This paper is concerned with long-time dynamics of semilinear wave equations of the form 
\begin{eqnarray} \label{P-ep}
\left\{
\begin{array}{ll}
\partial_t^2 u - \Delta u + a(x) g(\partial_t u) + f(u) = \varepsilon h(x) \;\; \mbox{in} \;\; \Omega \times \mathbb{R}^{+}, \smallskip \\
u=0 \;\; \mbox{on} \;\; \partial \Omega \times \mathbb{R}^{+}, \smallskip \\
u(0) = u_0, \;\; \partial_t u(0) = u_1 \;\; \mbox{in} \;\; \Omega, 
\end{array} 
\right. 
\end{eqnarray}
defined in a bounded domain $\Omega$ of $\mathbb{R}^3$ with smooth boundary $\partial \Omega$. 
We consider this problem with three distinguished features, namely, locally distributed damping, nonlinearity with critical Sobolev growth and external force with a parameter $\ep$. Here critical Sobolev growth means that $|f(u)|$ 
growths at most like $|u|^3$. As we will see, for a variety of $a,f,g,h$, problem \eqref{P-ep} has  
a unique finite energy solution in $C(\mathbb{R}^{+}, \mathcal{H})$, 
where $\mathcal{H} = H^1_0(\Omega) \times L^2(\Omega)$. Then the solution operator of \eqref{P-ep} defines a $C^0$-semigroup $\{S_{\ep}(t)\}_{t \ge 0}$ on $\mathcal{H}$.     

The existence of global attractors for wave equations with critical nonlinearity was firstly established by Arrieta, Carvalho and Hale \cite{ach}. They proved the existence of regular attractors in a framework of linear full damping $\partial_t u$, that is, with $a(x)=1$ and $g(s)=s$. Later, the existence of global attractors for wave equations with locally distributed damping was established by Feireisl and Zuazua \cite{fz}. 
They considered a nonlinear damping term $a(x)g(\partial_t u)$ localized in a collar of $\Omega$,  
that is, the support of $a(x)$ contains $\omega = \Omega \cap \mathcal{O}$, where $\mathcal{O}$ is some open neighborhood of $\partial \Omega$ in $\mathbb{R}^3$. The question of whether such attractors have finite fractal dimension was finally established by Chueshov, Lasiecka and Toundykov \cite{clt}. There, existence of regular finite dimensional attractors was proved for a more general damping region $\omega$, satisfying an observability condition (see Figure \ref{control}). Both damping regions $\omega$ in \cite{clt,fz} satisfy the geometric control condition (GCC) which asserts that every ray of geometric optics within $\Omega$ must reach the control region. See Figure \ref{GCC} and e.g. \cite{blr92,ralston,rauch-taylor}. 

\begin{figure}[htb] 
	\begin{center}
		\includegraphics[scale=0.4]{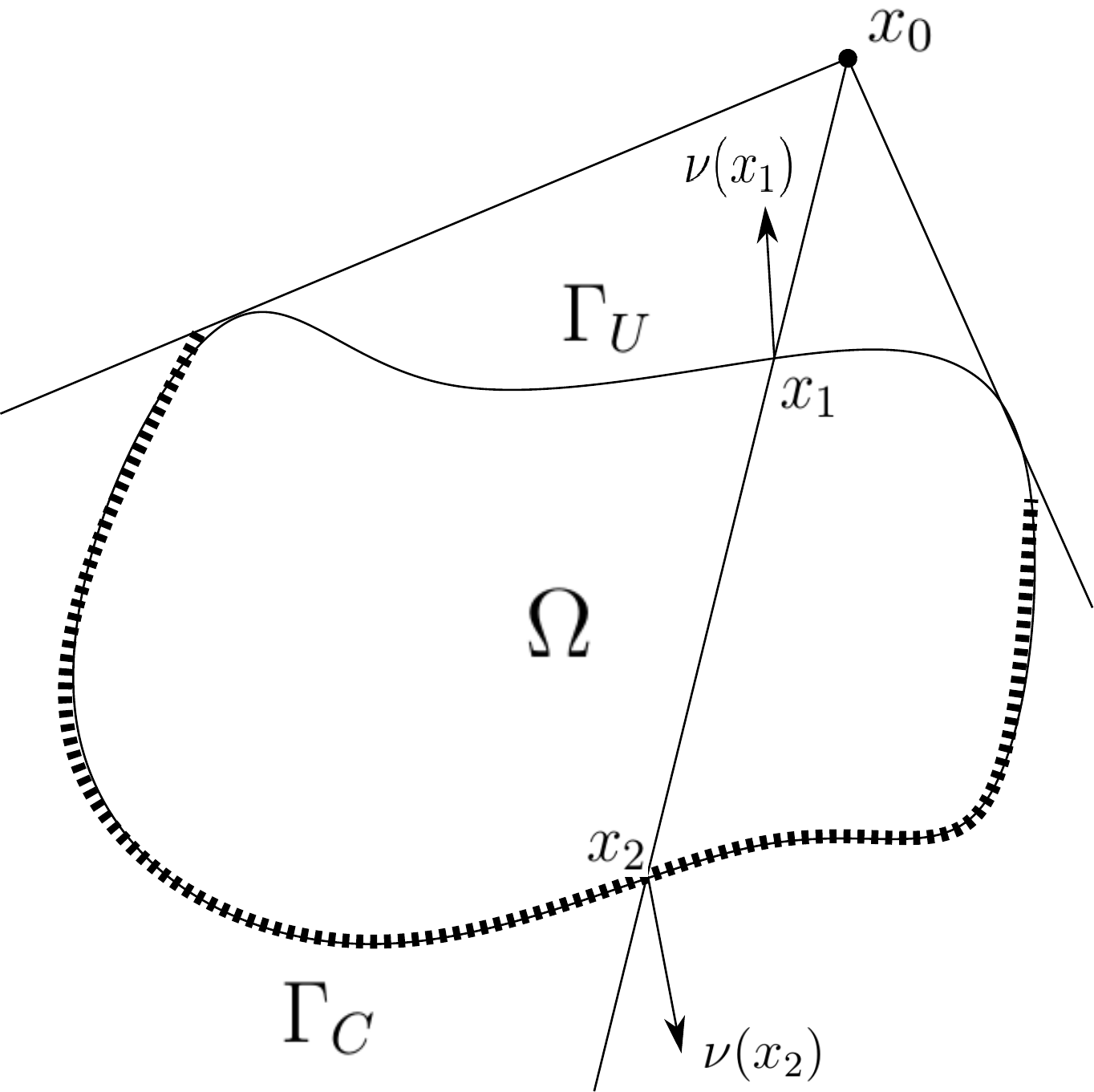} 
	\end{center}
	\caption{The sign of $(x-x_0) \cdot \nu(x)$ allows us to divide $\partial \Omega$ into an uncontrollable part $\Gamma_U$ and a controllable part $\Gamma_C$, where $x_0$ serves as an observer.} 
	\label{control}
\end{figure}

The results in \cite{ach,clt,fz} consider problem \eqref{P-ep} with external force $h=0$. Our objective is to assume $h\in L^2(\Omega)$ and 
study continuity aspects of attractors with respect to the parameter $\varepsilon \in [0, 1]$. We also study the existence of exponential attractors. 

The main contributions of the present paper are summarized as follows. 

$(a)$ The existence of attractors $\mathcal{A}_{\ep}$ for the system $(\mathcal{H},S_{\ep}(t))$, defined by \eqref{P-ep}, can be justified with \cite{clt}. However, since the damping term is in general effective only on a neighborhood of $\partial \Omega$, 
it is not easy to estimate the size of a bounded absorbing set. In other words, given an initial value $z$ in a bounded set $B \subset \mathcal{H}$, show the existence of a constant $C_{fh}>0$, independent of $B$, such that 
\begin{equation} \label{estimate-abs}
\| S_{\ep}(t) z \|_{\mathcal{H}} \le C_B e^{-\gamma t} + C_{fh}, \;\; \ep \in [0,1], \; t \ge 0. 
\end{equation}
Such an estimate would promptly show that attractors $\mathcal{A}_{\ep}$ 
are uniformly bounded. Indeed, in \cite{clt} the authors showed dissipativeness of the system by exploring 
its gradient structure combined with a unique continuation theorem \cite{toundykov}. 
In \cite{fz}, the authors proved existence of a bounded absorbing 
set by using a contradiction argument combined with a unique continuation theorem \cite{ruiz}, without 
showing an estimate like \eqref{estimate-abs}. Here, in Theorem \ref{thm-uniform}, we prove that attractors $\mathcal{A}_{\ep}$ are bounded uniformly with respect to $\ep \in [0,1]$. 

$(b)$ Once proved that attractors $\mathcal{A}_{\ep}$ are uniformly bounded, 
we can apply a recent result by Hoang, Olson and Robinson \cite{hoang} to study the continuity of 
$\ep \mapsto \mathcal{A}_{\ep}$ with respect to the Hausforff metric. Then, we prove that 
$\mathcal{A}_{\ep}$ is continuous on any $\ep \in J$, a dense {\it residual} subset of $[0,1]$.   
In addition, we show that upper-semicontinuity of $\mathcal{A}_{\ep}$ with respect to the parameter holds 
for all $\ep \in [0,1]$. See Theorems \ref{thm-lower} and \ref{thm-upper}. These results were not considered before for waves equations with locally distributed damping.

$(c)$ We recall that an exponential attractor for a system $(\mathcal{H}, S(t))$ 
is a compact set $\mathcal{M} \subset \mathcal{H}$ that is forward invariant, has finite fractal dimension, 
and exponentially attracts bounded sets of $\mathcal{H}$. Exponential attraction means that for any bounded set $D$, 
there exists a constant $\gamma_D>0$ such that   
$$
\lim_{t \to +\infty} e^{\gamma_D t} {\rm dist}_{\mathcal{H}}(S(t)D , \mathcal{M}) = 0,
$$  
where ${\rm dist}_{\mathcal{H}}$ represents the Hausdorff semi-distance in $\mathcal{H}$ (e.g. Eden et al \cite{efnt}). In general, establishing existence of exponential attractors for nonlinear wave equations 
is a difficult task. Indeed, most results are related to parabolic like equations. However, we shall prove the existence of so-called {\it generalized} exponential attractors (cf. Chueshov and Lasiecka \cite{CL-yellow}).  
This differs from the former one by requiring that the attractor have finite fractal dimension in a possibly larger space 
$\widetilde{\mathcal{H}}$ containing $\mathcal{H}$. See Theorem \ref{thm-exp}. 

\begin{figure}[htb] 
\begin{center}
\includegraphics[scale=0.4]{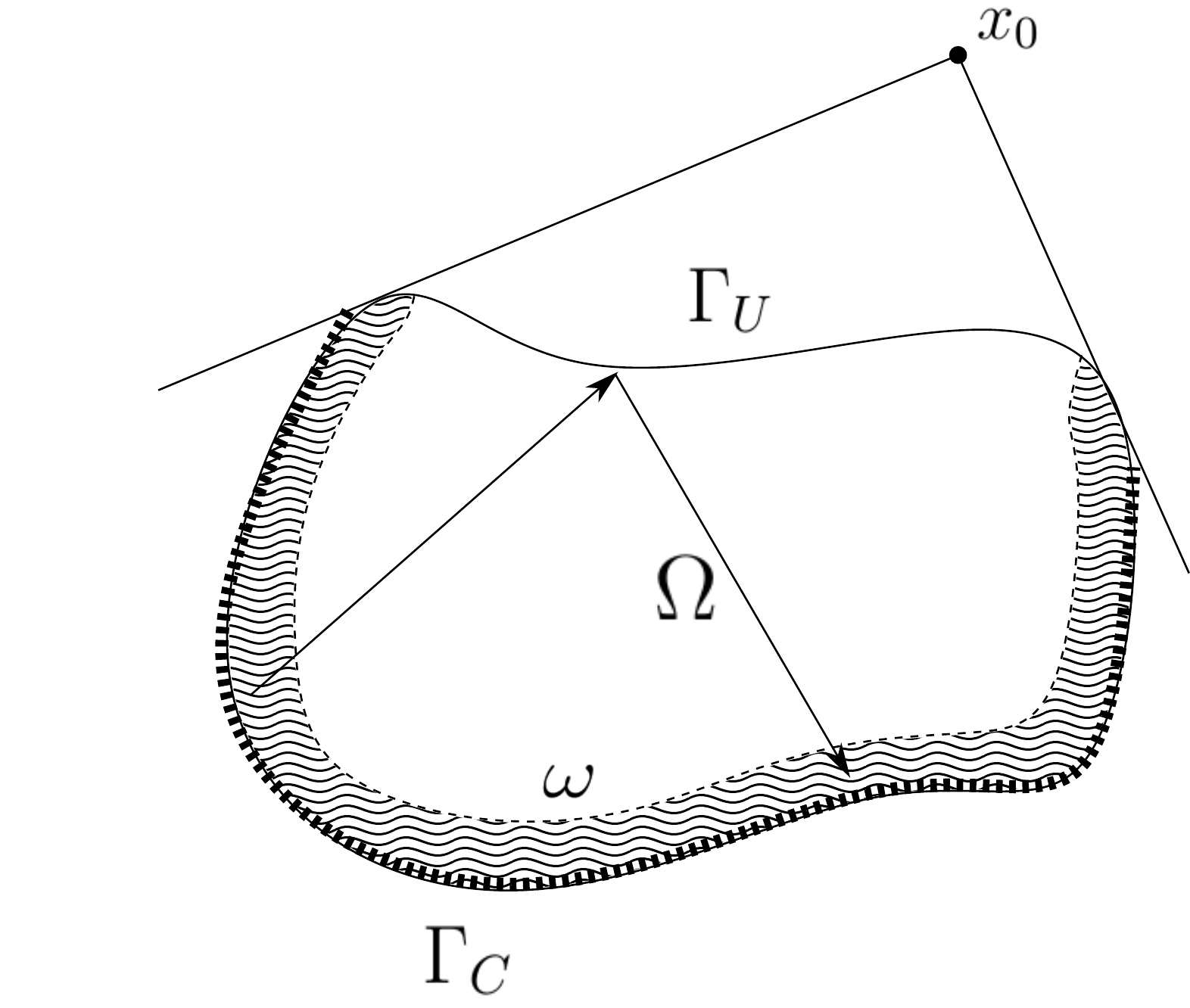} 
\end{center}
\caption{The control region $\omega$ satisfies (GCC). Any ray of geometric optics inside $\Omega$ hits $\omega$.} 
\label{GCC}
\end{figure}

\section{Well-posedness and global attractors} 
Our study is based in part on the results presented in \cite{clt} concerned with problem 
\eqref{P-ep}, with $h=0$, in the energy space 
$$
\mathcal{H} = H^1_0(\Omega) \times L^2(\Omega).
$$ 
The norm in $\mathcal{H}$ is given by  
$$
\| (u,v)\|_{\mathcal{H}}^2 = \|\nabla u\|_2^2 + \| v \|_2^2, 
$$
where $\| \cdot \|_{p}$ denotes $L^p(\Omega)$-norms. The existence of strong solutions is placed in  
$$
\mathcal{H}_{1} = (H^2(\Omega) \cap H^1_0(\Omega)) \times H^1_0(\Omega).
$$ 
We shall freely use standard notations and properties of Sobolev spaces as in e.g. \cite{hale,temam}.  

\subsection{Assumptions} 
Let $\Omega$ be a bounded domain of $\mathbb{R}^3$ with smooth boundary $\partial \Omega$.  
We consider nonlinear structural forces $f \in C^2(\Omega)$ satisfying 
\begin{equation} \label{f1}
|f''(s)|\le c_f(1+|s|), \;\; s \in \mathbb{R}
\end{equation}
and
\begin{equation} \label{f2}
\limsup_{|s|\to \infty}\frac{f(s)}{s} > - \lambda_1,
\end{equation}
where $\lambda_1>0$ is the principal eigenvalue of the Laplacian operator $-\Delta$ with Dirichlet boundary condition. 
With respect to the damping term, we assume $g \in C^1(\mathbb{R})$ with $g(0)=0$ and such that  
\begin{equation} \label{g}
g(s)s \ge 0 \;\; \mbox{and} \;\; m\le g'(s)\le M, \;\; s \in \mathbb{R}, 
\end{equation}
for some constants $m,M >0$. For the external force we take  
\begin{equation} \label{h}
h \in L^2(\Omega). 
\end{equation}
Finally, for the control/damping region we assume there exists a point $x_0 \in \mathbb{R}^3 \setminus \overline{\Omega}$ such that a (uncontrolled) part of the boundary 
$$
\Gamma_U = \{x \in \partial \Omega \, | \, (x-x_0)\cdot \nu(x) \le 0 \} 
$$
is nonempty, where $\nu$ denotes the outward normal vector on $\partial \Omega$. 
We also define an open connected (controlled) part of $\Gamma_C$ of $\partial \Omega$ that satisfies $\Gamma_C \cup \Gamma_U = \partial \Omega$,  possibly overlapping  (see figure \ref{control}). Then we can define the control/damping region 
\begin{equation} \label{omega}
\omega = \{x \in \Omega \, | \, {\rm dist}(x, \Gamma_C) \le \delta \}, 
\end{equation}
for some $\delta >0$. Finally we take $a \in L^{\infty}(\Omega)$, nonnegative, such that for some $a_0>0$, 
\begin{equation} \label{x2}
a(x) \ge a_0  \;\; \mbox{on} \;\; \omega. 
\end{equation}

\begin{remark}	$(a)$ The boundedness of $g'(s)$ in the assumption \eqref{g} is not necessary for global existence. As shown in \cite{clt}, this is essential for the proof that attractors have finite fractal dimension. $(b)$ We observe that $\omega$ constructed in \eqref{omega} satisfies the geometric control condition (GCC). See figure \ref{GCC}. \qed
\end{remark}	
	
\subsection{Well-posedness} 
We can write problem \eqref{P-ep} as a Cauchy problem
\begin{equation}\label{P-U}
\partial_t U + \mathbb{A} U + \mathbb{F} U = \mathbb{H}, \quad U(0) = [u_0,u_1]^{\top}, 
\end{equation}
defined in $\mathcal{H}$, 
where
$$
U = \begin{bmatrix}
u  \\ \partial_t u  \\
\end{bmatrix}, \quad 
\mathbb{A} = \begin{bmatrix}
0 & -Id  \\ -\Delta & a(x)g(\cdot)  \\
\end{bmatrix}, \quad 
\mathbb{F} = \begin{bmatrix}
0 & 0  \\ f(\cdot) & 0 \\
\end{bmatrix}, \quad 
\mathbb{H} = \begin{bmatrix}
0  \\ \varepsilon h(x)  \\
\end{bmatrix},
$$
and $D(\mathbb{A}) = \mathcal{H}_1$. 

\begin{theorem} [Well-posedness \cite{clt}] \label{thm-wp} Suppose that hypotheses \eqref{f1}-\eqref{x2} are satisfied and $\ep \in [0,1]$. Then we have:  
\begin{enumerate}
\item If $(u_0,u_1) \in \mathcal{H}$ then problem \eqref{P-ep} has a unique solution  
$$
u \in C(\mathbb{R}^+;H^1_0(\Omega))\cap C^1(\mathbb{R}^+;L^2(\Omega)).  
$$
\item If $(u_0,u_1) \in \mathcal{H}_1$, then above solution has stronger regularity  
$$
u \in C(\mathbb{R}^+;H^2(\Omega)\cap H^1_0(\Omega))\cap C^1(\mathbb{R}^+;H^1_0(\Omega)).
$$
\item Let us denote the solution operator by $S_{\ep}(t)$. Then given $T>0$ and a bounded set $B$ of $\mathcal{H}$,  
there exists a constant $C_{BT}>0$ such that
\begin{equation} \label{S1} 
\|S_{\varepsilon}(t)z^1-S_{\varepsilon}(t)z^2\|_{\mathcal{H}}\le C_{BT} \|z^1-z^2\|_{\mathcal{H}},
\end{equation}
for all $t \in [0,T]$ and $z^1,z^2 \in B$.
\end{enumerate}
\end{theorem}

\begin{remark} 
$(a)$ Theorem \ref{thm-wp} is a known result. A detailed proof with $h=0$ is presented in \cite{clt}. 
Since the external force $h=h(x)$ is not time-dependent it does not change the arguments in \cite{clt}.
Essentially, under the hypotheses, $\mathbb{A}$ is maximal monotone in $\mathcal{H}$, and $\mathbb{F}$ is locally Lipschitz in $\mathcal{H}$ since 
the nonlinear perturbation $f(u)$ is locally Lipschitz from $H^1_0(\Omega)$ to $L^2(\Omega)$. 
$(b)$ The continuity estimate \eqref{S1} shows that solution operator $S_{\varepsilon}(t)$ is a $C^0$-semigroup on $\mathcal{H}$.  \qed 
\end{remark}

We end this subsection with some remarks on the energy of the system. Given any solution $u \in C(\mathbb{R}^+;H^1_0(\Omega)) \cap C^1(\mathbb{R}^+;L^2(\Omega))$ of problem \eqref{P-ep} 
we define its total energy by setting  
$$ 
\mathcal{E}_{\ep}(t) = \frac{1}{2} \|(u(t),\partial_t u(t))\|^2_{\mathcal{H}} 
+ \int_{\Omega} F(u(t))\,dx - \ep \int_{\Omega} h(x) u(t) \, dx,
$$ 
where $F(u) = \int_{0}^{u}f(r) \,dr$. 

\begin{lemma} 
There exist positive constants $\nu_1,c_1,c_2$, depending on $f$ and $h$, such that 
\begin{equation} \label{desiEnergy}
\nu_1 \|(u,\partial_t u)\|^2_{\mathcal{H}} - c_1 \le \mathcal{E}_{\ep} \le c_2(1+\|(u,\partial_t u)\|^4_{\mathcal{H}}), 
\;\; \forall \,\ep \in [0,1].
\end{equation} 
In addition, 
\begin{equation} \label{Energy-prime}
\frac{d}{dt} \mathcal{E}_{\ep}  = - \int_{\Omega} a(x)g(\partial_t u)\partial_t u \, dx, 
\;\; \forall \, \ep \in [0,1].
\end{equation}

\begin{proof} The proof is standard. Indeed, the first and second inequalities in \eqref{desiEnergy} follow 
from assumptions \eqref{f2} and \eqref{f1} respectively, for any $\ep \in [0,1]$. 
We sketch a proof of the first inequality since it will be used some times. 
Indeed, from assumption \eqref{f2} we can choose $\delta >0$ (small) such that 
$\liminf f(s)s^{-1} \ge - (\lambda_1- \delta)$ as $|s| \to \infty$. This implies the existence of a constant $k_F>0$ such that 
$$
F(s) \ge  \frac{-\lambda_1 + \delta}{2} s^2 - k_F , \quad \forall \, s \in \mathbb{R}.  
$$
Then 
$$
\int_{\Omega} F(u) \, dx \ge \frac{1}{2} \left( \frac{-\lambda_1 + \delta}{\lambda_1} \right) \|\nabla u\|_2^2 - k_F |\Omega|.   
$$
From this we infer the first inequality in \eqref{desiEnergy} by taking $\nu_1 = \delta/(2 \lambda_1)$. 
To obtain the energy estimate \eqref{Energy-prime} we integrate first equation of \eqref{P-ep} multiplied by $\partial_t u$. 	
\end{proof}
\end{lemma}

\begin{remark} \label{rem-E0}
Since the energy $\mathcal{E}_{\ep}(t)$ is not increasing, using \eqref{desiEnergy} we can estimate 
$\|S_{\ep}(t)z\|_{\mathcal{H}}$ by the size of initial value $\| z\|_{\mathcal{H}}$. In particular, 
all solution trajectories with initial value in a bounded set $B$ remain uniformly bounded. Indeed,   
\begin{align*} 
\|S_{\ep}(t)z\|_{\mathcal{H}}^2  
& \le \frac{1}{\nu_1}\left( \mathcal{E}_{\ep}(0) + c_1 \right) \\
& \le \frac{1}{\nu_1}\left( c_2 \big(1+ \|z\|_{\mathcal{H}}^4 \big) + c_1 \right),  
\end{align*}
independently of $\ep \in [0,1]$.  \qed
\end{remark}

\subsection{Global attractors}
In this section we recast the main result in \cite{clt} on the existence of regular finite dimensional attractors for critical waves with locally distributed damping. We recall that a global attractor for a system $(X,S(t))$, where $X$ is a complete metric space and $S(t)$ is a $C^0$-semigroup, is a nonempty fully invariant compact set 
$A \subset X$ that attracts bounded sets of $X$. The fractal (box-counting) dimension of a compact set 
$A$ is defined 
$$
{\rm dim}_{f}^X (A) = \limsup_{r \to 0} \frac{\ln(N_r)}{\ln(1/r)}, 
$$
where $N_r$ is the minimal number of closed balls $\overline{B}(0,2r)$ necessary to cover $A$.  
See e.g. \cite{babin-vishik,hale,lady,temam} or \cite[Chapter 7]{CL-yellow}.   

\begin{theorem} [Global attractors \cite{clt}] \label{thm-attractor} Suppose that hypotheses \eqref{f1}-\eqref{x2} are satisfied and $\ep \in [0,1]$. 
Then:
\begin{enumerate}
\item The dynamical system $(\mathcal{H},S_{\ep}(t))$ corresponding to problem \eqref{P-ep} possesses a global attractor $\mathcal{A}_{\ep}$ with finite fractal dimension. 

\item Let $\mathcal{N}_{\ep}$ denote the set of stationary solutions of the problem \eqref{P-ep}. Then we have 
$$
\mathcal{A}_{\ep} = \mathbb{W}^u(\mathcal{N}_{\ep}), 
$$
the unstable manifold emanating from the stationary points. 
 
\item The attractor $\mathcal{A}_{\ep}$ has higher regularity $\mathcal{H}_1$ 
and its full trajectories $(u(t),\partial_t u(t))$, $t \in \mathbb{R}$, satisfy  
\begin{equation} \label{smooth} 
\|(\partial_t u(t),\partial^2_t u(t))\|_{\mathcal{H}}\le \mathcal{Q}(\|\mathcal{A}_{\varepsilon}\|_{\mathcal{H}}),
\end{equation}
where $\mathcal{Q}(\cdot)$ is a generic increasing positive function not depending on $\ep$ or on a particular trajectory.
\end{enumerate}
\end{theorem}

\begin{remark}
The proof the theorem, with $h=0$, was presented in \cite{clt}. Since $h$ is an autonomous perturbation, the proof of the theorem with $h$ follows with same methods and arguments. Firstly, it is shown that the system has a gradient structure. This is done by using a unique continuation theorem in \cite{fu,toundykov}. Then, a difficult part is to show asymptotic regularity/compactness of the system. To this end, new observability inequalities, trough Carleman estimates \cite{fu,LTY,LTZ}, are obtained in order to prove a stability inequality. For the reader's convenience we shall prove the gradient structure in Lemma \ref{lemma-gradient} below.   \qed    	
\end{remark}

\begin{remark} After the apparition of \cite{clt}, Blair, Smith and Sogge \cite{blair} proved new 
Strichartz type estimates for wave equations on manifolds with boundary. This provided new tools to study wave equations in bounded domains of $\mathbb{R}^3$ with quintic nonlinearities (rather than cubic). With respect to 
wave equations with locally distributed damping, Joly and Laurent \cite{joly-laurent} considered linear damping and sub-quintic nonlinearities ($|f(x,u)| \approx |u|^{5-\delta}$). Their results were presented in a formalism of Riemannian geometry, assuming that the control region $\omega$ satisfies (GCC) only. To establish existence of global attractors, they proved a proper unique continuation theorem based on the one by Robbiano and Zuily \cite{rz}, that requires $f=f(x,u)$ to be analytic. 
In the framework of \cite{joly-laurent}, regularity and fractal dimension of attractors are open questions. We refer the reader to, e.g. \cite{cel,conti,ksz,ma,mei,yang} for some recent works on attractors for wave equations with cubic and quintic nonlinearities. \qed
\end{remark}

\section{Uniformly bounded attractors} 

In this section we prove that global attractors $\mathcal{A}_{\ep}$ given by Theorem \ref{thm-attractor} are uniformly bounded with respect to $\ep \in [0,1]$. We recall that a 
mapping $\Psi : \mathcal{H} \to \mathbb{R}$ is called a strict Lyapunov function 
for a system $(\mathcal{H},S(t))$ if the functional $\Psi(S(t)z)$ is non-increasing with respect to $t\ge 0$, for any $z \in \mathcal{H}$. Moreover, if $\Psi(S(t)z) = z$ for all $t\ge 0$, then $z$ is a fixed point of $S(t)$. A dynamical system is called gradient if it possesses a strict Lyapunov function.

\begin{lemma}[Gradient system \cite{clt}] \label{lemma-gradient}
Suppose that hypotheses of Theorem \ref{thm-wp} hold. Then the total energy $\mathcal{E}_{\ep}(t)$ is 
a strict Lyapunov functional.     
\end{lemma}

\begin{proof} 
We sketch a proof for the reader's convenience. Let us define 
\begin{equation} \label{Lyapunov}
\Psi_{\ep}(S_{\ep}(t)z) := \frac{1}{2} \|(u(t),\partial_t u(t))\|^2_{\mathcal{H}} 
+ \int_{\Omega} F(u(t))\,dx - \ep \int_{\Omega} h(x) u(t) \, dx,
\end{equation} 
where $(u(t),\partial_t u(t))$ is the solution of problem \eqref{P-ep} with initial data $z$. 
From \eqref{Energy-prime} we see that $\Psi_{\ep}(S_{\ep}v(t)z)$ is non-increasing. Now, if $\Psi_{\ep}(S_{\ep}(t)z)$ is stationary for some $z$, it follows that    
$$
\int_{\Omega} a(x)g(\partial_t u)\partial_t u \, dx = 0, \quad t \ge 0. 
$$ 
Using \eqref{g} we infer that 
$$
\int_{\omega} |\partial_t u|^2 \, dx =0  \;\; \mbox{and} \;\;  
\int_{\Omega} a(x) |g(\partial_t u)|^2 \, dx = 0, \quad t \ge 0.
$$
In particular, for any $T>0$,  
$$
\partial_t u = 0 \;\; \mbox{a.e. in} \;\; \omega \times (0,T) \quad \mbox{and} \quad a(x) g(\partial_t u) =0 \;\; \mbox{a.e. in} \;\; \Omega \times (0,T).  
$$
Therefore, the semiflow $S_{\ep}(t)z=(u(t),\partial_t u(t))$ is 
a solution of 
\begin{eqnarray*} 
\left\{
\begin{array}{ll}
\partial_t^2 u - \Delta u + f(u) = \varepsilon h(x) \;\; \mbox{in} \;\; \Omega \times (0,T), \smallskip \\
u=0 \;\; \mbox{on} \;\; \partial \Omega \times (0,T), \smallskip \\
\partial_t u = 0 \;\; \mbox{in} \;\; \omega \times (0,T), \smallskip \\
u(0) = u_0, \;\; \partial_t u(0) = u_1 \;\; \mbox{in} \;\; \Omega \times (0,T). 
\end{array} 
\right. 
\end{eqnarray*}
By density arguments, we can assume $(u,\partial_t u)$ is a strong solution in $\mathcal{H}_1$ and 
then, differentiating the equation and putting $w = \partial_t u$,  
we obtain 
\begin{eqnarray} \label{dual}
\left\{
\begin{array}{ll}
\partial_t^2 w - \Delta w + f'(u) w = 0 \;\; \mbox{in} \;\; \Omega \times (0,T), \smallskip \\
w = 0 \;\; \mbox{on} \;\; \partial \Omega \times (0,T), \smallskip \\
w = 0 \;\; \mbox{in} \;\; \omega \times (0,T).
\end{array} 
\right. 
\end{eqnarray}
Upon \eqref{dual} we apply a unique continuation theorem which is compatible with the geometric conditions of $\omega$ (e.g. \cite[Theorem 2.2]{fu}) to conclude that 
$w = \partial_t u=0$ over $\Omega \times \mathbb{R}^{+}$. Therefore $z=(u_0,0)$ is a fixed point of $S_{\ep}(t)$. 
\end{proof}

\begin{theorem} \label{thm-uniform}
Under the hypotheses of Theorem \ref{thm-attractor}, there exists a constant $R >0$ such that 
\begin{equation} \label{uniform-bound}
\mathcal{A}_{\ep} \subset B(0,R), \;\; \forall \, \ep \in [0,1].
\end{equation}
In addition, there exists a constant $C_{R}>0$ such that 
\begin{equation} \label{smooth2}
\|(\partial_t u(t),\partial^2_t u(t))\|_{\mathcal{H}} \le C_{R}, \;\; \forall \ep \in [0,1], \;\ t \in \mathbb{R},
\end{equation}
for any full trajectory $\{(u(t),\partial_t u(t)) \, | \, t \in \mathbb{R}\}$  that belongs to $\mathcal{A}_{\ep}$.  
\end{theorem}

\begin{proof} From Lemma \ref{lemma-gradient} we know that system $(\mathcal{H}, S_{\ep}(t))$ is gradient 
by taking the energy functional as a Lyapunov function \eqref{Lyapunov}.  
	
Next, we recall a property of gradient systems that asserts the following: if $\mathcal{A}$ is the global attractor of a gradient system with Lyapunov function $\Psi$, and bounded set of stationary points $\mathcal{N}$, then  
$\sup \{\Psi(z) \, | \, z \in \mathcal{A} \} \le \sup \{ \Psi(z) \, | \, z \in \mathcal{N}\}$. See e.g. \cite[Remark 7.5.8]{CL-yellow}.  
In our context, we have 
\begin{equation} \label{des}
\sup\{\Psi_{\ep}(z) \,| \, z \in \mathcal{A}_{\varepsilon} \} \le \sup \{ \Psi_{\ep}(z) \, | \, z \in \mathcal{N}_{\ep} \}, 
\end{equation}
where $\mathcal{N}_{\ep}$ denotes the set of fixed points of $S_{\ep}(t)$. 

Now, we observe that $\mathcal{N}_{\ep}$ must be uniformly bounded. Indeed, if $z \in \mathcal{N}_{\ep}$ then $z=(u,0)$ where $u$ 
is a weak solution of the stationary problem 
$$
-\Delta u + f(u) = \ep h(x) \;\; \mbox{in} \;\; H^1_0(\Omega). 
$$
Using assumption \eqref{f2} and Poincar\'e inequality we obtain a constant $C>0$ such that 
$$
\| \nabla u \|_{2}^2 = \int_{\Omega} (\ep h - f(u) ) u \, dx \le \frac{1}{2} \| \nabla u \|_{2}^2 + C (1 + \| h\|_2^2).
$$
Therefore $\| \nabla u\|_{2} \le C_{fh}$ for some constant $C_{fh}>0$, independently of $\ep \in [0,1]$. 
This shows that $\mathcal{N}_{\ep}$ is uniformly bounded with respect to $\ep$. In particular, taking into account the definition of $\Psi_{\ep}$ in \eqref{Lyapunov} and the second inequality in \eqref{desiEnergy},  
we see that 
$$ 
\Psi_{\ep}(z) \le c_2 (1 + C_{fh}^4), \;\; \forall \, z \in \mathcal{N}_{\ep}, \;\; \forall \, \ep \in [0,1].  
$$
On the other hand, using the first inequality in \eqref{desiEnergy}, we see that 
$$ 
\nu_1 \| z \|_{\mathcal{H}}^2 - c_1 \le \Psi_{\ep}(z), \;\; \forall \, z \in \mathcal{A}_{\ep}, 
\;\; \forall \, \ep \in [0,1].  
$$
Combining last two estimates with \eqref{des} we obtain the uniform bound \eqref{uniform-bound} by taking 
$R^2 = \nu^{-1} (1+ c_1 +c_2 + c_2 C_{fh}^4)$. Finally, taking $C_{R} = \mathcal{Q}(R)$, we see from \eqref{smooth} that estimate \eqref{smooth2} holds.  
\end{proof}

\begin{remark} \label{rem-absorb}
We note that from \eqref{uniform-bound}, taking $R_0>R$, the ball $\mathcal{B} = B(0,R_0)$ is a bounded absorbing set of $(\mathcal{H},S_{\ep}(t))$, uniformly on $\ep$. Given a bounded set $B \subset \mathcal{H}$, there exists a entrance time $t_B>0$ such that 
$S_{\ep}(t)B \subset \mathcal{B}$, if $t \ge t_B$, for any $\ep \in [0,1]$.   \qed
\end{remark}

\section{On the continuity of attractors}
Let $\mathcal{A}_{\lambda}$ be a family of global attractors for a system $(X,S_{\lambda}(t))$, 
where $\lambda$ belongs to a complete metric space $\Lambda$. We say that 
$\mathcal{A}_{\lambda}$ is upper semicontinuous on $\lambda_0 \in \Lambda$ if 
$$
\lim_{\lambda \to \lambda_0} {\rm dist}_X(\mathcal{A}_{\lambda},\mathcal{A}_{\lambda_0}) = 0,  
$$
where ${\rm dist}_X(A,B) = \sup_{a \in A} \inf_{b \in B} d(a,b)$ denotes the Hausdorff semi-distance in $X$. 
Analogously, if we commutate $\mathcal{A}_{\lambda}$ and $\mathcal{A}_{\lambda_0}$ 
in the above limit, then we say that 
$\mathcal{A}_{\lambda}$ is lower semicontinuous on $\lambda_0 \in \Lambda$. 
Then $\mathcal{A}_{\lambda}$ is continuous on $\lambda_0 \in \Lambda$ if  
$$
\lim_{\lambda \to \lambda_0} {\rm d}_{X}(\mathcal{A}_{\lambda},\mathcal{A}_{\lambda_0}) = 0,  
$$
where ${\rm d}_{X}(A,B)= \max \big\{{\rm dist}_X(A,B),{\rm dist}_X(B,A) \big\}$ is the Hausdorff metric in $X$.

While it is more or less standard to check upper semicontinuity of attractors for a large class of dissipative systems, the proof of lower semicontinuity is much more involving (cf. \cite{hr1988,hr1989}). 

We shall use a recent result by Hoang, Olson and Robinson \cite{hoang} on   
the continuity of attractors with respect to a parameter. Their results extend earlier ones by Babin and Pilyugin \cite{Babin-Pilyugin}. 
Accordingly, let $S_{\lambda}(t)$ be a family of parametrized semigroups defined on $X$, 
with $\lambda$ in a complete metric space $\Lambda$. Assume that 
\begin{itemize}
\item[$(c1)$] The system $(X,S_{\lambda}(t))$ has a global attractor $\mathcal{A}_{\lambda}$ for every $\lambda \in \Lambda$,

\item[$(c2)$] There exists a bounded set $\mathcal{B} \subset X$ such that 
$\mathcal{A}_\lambda \subset \mathcal{B}$ for every $\lambda \in \Lambda$,

\item[$(c3)$] For $t>0$, $S_{\lambda}(t)x$ is continuous in $\lambda$, uniformly for $x$ in bounded subsets of $X$.
\end{itemize}
Then $\mathcal{A}_{\lambda}$ is continuous on all $\lambda \in J$ where $J$ is a ``residual'' set dense in $\Lambda$. See \cite[Theorem 5.2]{hoang}. 

\begin{theorem}[Continuity on a residual dense subset] \label{thm-lower}
In the context of Theorem \ref{thm-attractor} there exists a set $J$ dense in $[0,1]$ 
such that $\mathcal{A}_{\ep}$ is continuous with respect to any parameter $\ep_0 \in J$,  
that is,   
\begin{equation} \label{J}
\lim_{\ep \to \ep_0} {\rm d}_\mathcal{H} (\mathcal{A}_{\ep},\mathcal{A}_{\ep_0}) = 0, 
\;\; \forall \, \ep_0 \in J.
\end{equation}
\end{theorem}

\begin{proof}
We shall apply \cite[Theorem 5.2]{hoang} to the context of Theorem \ref{thm-attractor}. 
Then above assumption $(c1)$ holds. The assumption $(c2)$ also holds because 
of the uniform bound \eqref{uniform-bound} in Theorem \ref{thm-uniform}. 
It remains to prove condition $(c3)$. 

Let $B$ be a bounded set of $\mathcal{H}$. Given $\ep_1,\ep_2 \in [0,1]$ and $z \in B$, 
let us denote 
$$
S_{\ep_1}(t)z = (u(t),\partial_t u(t)) \;\; \mbox{and} \;\; 
S_{\ep_2}(t)z = (v(t),\partial_t v(t)). 
$$
Then $w=u-v$ is a solution of 
$$ 
\left \{
\begin{array}{ll}
\partial^2_{t} w-\Delta w + a (g(\partial_t u)-g(\partial_t v)) 
+ f(u)-f(v) = (\ep_1-\ep_2) h \;\; \mbox{in} \;\; \Omega \times \mathbb{R}^{+}, \smallskip \\
w=0 \;\; \mbox{on} \;\; \partial \Omega \times \mathbb{R}^{+}, \smallskip \\
w(0) = \partial_t w(0) = 0 \;\; \mbox{in} \;\; \Omega, 
\end{array} 
\right. 
$$ 
and consequently we have  
\begin{align} \label{geral}
\frac{1}{2}\frac{d}{dt} \|(w,\partial_t w)\|^2_{\mathcal{H}} 
= & \int_{\Omega}(f(v) - f(u) ) \partial_t w \, dx 
- \int_{\Omega} a (g(\partial_t u)-g(\partial_t v)) \partial_t w \, dx \nonumber \\ 
& + (\varepsilon_1-\varepsilon_2) \int_{\Omega} h\partial_t w \, dx.
\end{align}
By assumption \eqref{f1} we know that $|f'(s)| \le C (1 + |s|^2)$ for some $C >0$. Then,	
\begin{align*}
\int_{\Omega}(f(v)-f(u)) \partial_t w \, dx  
& \le C (1 + \| u \|_6^2 + \|v\|_6^2 ) \| w \|_6 \|\partial_t w\|_2 \\ 
& \le C (1 + \| \nabla u \|_2^4 + \| \nabla v\|_2^4 ) \| \nabla w \|_2^2 +  \|\partial_t w\|_2^2 . 
\end{align*}
From Remark \ref{rem-E0}, $\| \nabla u \|_2^4 + \| \nabla v\|_2^4$ are uniformly bounded, and hence there exists a constant $C_{B}>0$ such that  
$$
\int_{\Omega}(f(v)-f(u)) \partial_t w \, dx \le C_{B} \| (w, \partial_t w) \|_{\mathcal{H}}^2.  
$$
The assumptions \eqref{g} and \eqref{x2} imply that 
$$
- \int_{\Omega} a(g(\partial_t u)-g(\partial_t v)) \partial_t w \, dx \le 0.
$$
Also,  
$$
(\ep_1-\ep_2) \int_{\Omega} h \partial_t w \, dx \le |\ep_1-\ep_2|^2 \|h\|_2^2 + 4 \|(w,\partial_t w)\|^2_{\mathcal{H}}.
$$
Inserting above three estimates into \eqref{geral} we obtain  	
$$
\frac{d}{dt} \|(w,\partial_t w)\|^2_{\mathcal{H}} \le C_0 
\|(w,\partial_t w)\|^2_{\mathcal{H}} + |\ep_1-\ep_2|^2 \|h\|_2^2,
$$
where $C_0 = C_0(B,f,h)>0$ is independent of $\ep_1,\ep_2$. Then Gronwall inequality yields 
$$
\|(w,\partial_t w)\|^2_{\mathcal{H}} 
 \le e^{C_0 t} \|(w(0),\partial_t w(0)\|_{\mathcal{H}}^2 + C_0 (e^t-1) |\ep_1-\ep_2|^2 \|h\|_2^2 .
$$
Since $\|(w(0),\partial_t w(0)\|_{\mathcal{H}}=0$, we finally see that 
$$
\| S_{\ep_1}(t)z - S_{\ep_2}(t)z \|_{\mathcal{H}}  
\le \sqrt{C_0 (e^t-1)}\|h\|_2 |\ep_1-\ep_2|, \;\; t >0, 
$$
which shows $(c3)$. As a conclusion, by applying \cite[Theorem 5.2]{hoang}, there exists a dense set $J \subset [0,1]$ 
such that \eqref{J} holds. This ends the proof. 
\end{proof}

\begin{theorem}[Upper semicontinuity] \label{thm-upper} 
In the context of Theorem \ref{thm-attractor} the family of global attractors $\mathcal{A}_{\ep}$ 
is upper semicontinuous with respect to parameters $\ep$ in $[0,1]$, that is,
\begin{equation} \label{J2}
\lim_{\ep \to \ep_0} {\rm dist}_{\mathcal{H}}(\mathcal{A}_{\ep},\mathcal{A}_{\ep_0}) = \ep_0, 
\;\; \forall \, \ep_0 \in [0,1].
\end{equation}
\end{theorem}

\begin{proof} 
We argue by contradiction following \cite{hr1988}. 
Suppose that $\mathcal{A}_{\ep}$ is not upper semicontinuous at $\ep_0 \in [0,1]$. 
Then from \eqref{J2} there exists $\delta>0$ and sequences $\ep^n \to \ep_0$ 
and $z_0^n \in \mathcal{A}_{\ep^n}$ such that
\begin{equation} \label{J3} 
\inf_{y \in \mathcal{A}_{\ep_0}}\|z_0^n-y\|_{\mathcal{H}} \ge \delta, \; \; \forall \, n \in \mathbb{N}.
\end{equation}
Since the global attractors $\mathcal{A}_{\ep}$ are also characterized as  
$$
\mathcal{A}_{\ep} = \{z(0) \, | \, z \; \mbox{is a bounded full trajectory of $S_{\ep}(t)$} \},
$$
let $z^n = (u_n,\partial_t u_n)$ be a bounded full trajectory of $\mathcal{A}_{\ep^n}$ such that $z^n(0)=z_0^n$. 
Then by \eqref{smooth2} in Theorem \ref{thm-uniform}, we see that
\begin{equation} \label{conver1}
(u_n,\partial_t u_n) \;\; \mbox{is bounded in} \;\; W^{1,\infty}(\mathbb{R}, \mathcal{H}_1).
\end{equation}
From classical compactness arguments, we deduce the existence of a pair $(u,\partial_t u) \in  C(\mathbb{R};\mathcal{H}_1)$, 
such that, up to a subsequence,
\begin{equation} \label{conver2}
(u_n,\partial_t u_n) \to (u,\partial_t u) \;\; \mbox{ in } \;\; C([-T,T];\mathcal{H}), \;\; \forall \, T>0.
\end{equation}
Moreover, by (\ref{conver1}) and (\ref{conver2}), it follows that 
$$
\sup_{t \in \mathbb{R}}\| (u(t),\partial_t u(t)) \|_{\mathcal{H}} < \infty.
$$
We claim that $z=(u,\partial_t u)$ is a bounded full trajectory of the limiting semi-flow $S_{\ep_0}(t)$. 
Now, it is enough to show that $z$ is a full bounded trajectory for the problem \eqref{P-ep} with $\ep = \ep_0$, that is, 
\begin{equation} \label{eq}
\partial^2_{t} u - \Delta u + a(x)g(\partial_t u) + f(u) = \ep_0 h(x), \;\; \mbox{a.e.} \;\; t \in \mathbb{R}.
\end{equation}
Indeed, since $(u_n,\partial_t u_n)$ satisfies the equation
\begin{equation} \label{eq2}
\partial^2_{t} u_n - \Delta u_n + a(x)g(\partial_t u_n) + f(u_n) = \ep^n h(x), \;\; n \in \mathbb{N},
\end{equation}
we can proceed as in the verification of $(c3)$ in Theorem \ref{thm-lower} to conclude that  
(\ref{eq}) is the limit of \eqref{eq2} as $n \to \infty$. This is possible because, for any $n \in \mathbb{N}$, the control conditions in the damping region $\omega$ remains the same. Therefore 
$$
z(0) \in \mathcal{A}_{\ep_0} \;\; \mbox{and} \;\; \lim_{n \to \infty} \| z_0^n - z(0) \|_{\mathcal{H}} = 0, 
$$
which contradicts \eqref{J3}. 
\end{proof}

\section{Generalized exponential attractor} \label{sec-exp}

The objective of this section is to prove the existence of a generalized fractal exponential attractor for the dynamic system associated with problem \eqref{P-ep}. Let $X$ be a Hilbert space. 
Formally, cf. \cite[Definition 7.4.4]{CL-yellow}, a generalized (fractal) exponential attractor for a system 
$(X,S(t))$ is a compact set $\mathcal{A}^{\rm exp} \subset X$ that attracts exponentially bounded sets of $X$, is forward invariant, and have finite fractal dimension in a larger space ${\widetilde X} \supseteq X$. Our result reads as follows.   

\begin{theorem} \label{thm-exp}
Under the hypotheses of Theorem \ref{thm-attractor}, the system $(\mathcal{H},S_{\ep}(t))$ associated to problem \eqref{P-ep} possesses a generalized exponential attractor. 
\end{theorem}

The proof of Theorem \ref{thm-exp} will be completed at the end of this section. It needs the following abstract result. 

\begin{theorem} $($\cite[Theorem 7.4.2]{CL-yellow}$)$ \label{thm-lasi}
Let $M$ be a closed bounded set of a separable Hilbert space $H$ and 
suppose that $V: M \to M$ is a Lipschitz mapping. In addition, suppose there exist compact seminorms $n_1$ and $n_2$ on $H$ such that 
$$ 
\|V_1 v_1-Vv_2\| \le \eta \|v_1-v_2\|+K(n_1(v_1-v_2)+n_2(Vv_1-Vv_2)), 
$$
for any $v_1,v_2 \in M$, where $0<\eta<1$ and $K>0$ are constants. Then, for any $\theta \in (\eta,1)$ 
there exists a positively invariant compact set $A_{\theta}\subset M$, of finite fractal dimension, satisfying
$$
{\rm dist}_{M}(V^k M,A_{\theta})\le r\theta^k, \;\; k \in \mathbb{N}, 
$$
for some $r>0$.
\end{theorem}

\begin{remark} To apply Theorem \ref{thm-lasi} in the context of problem \eqref{P-ep} we shall apply rather technical arguments. 
They rely on some results of \cite{clt} mainly related to Carleman estimates in order to absorb some lower order terms generated by, for instance, the integral   
$$
\int_{\Omega}(f(u)-f(v))(\partial_t u - \partial_t v)dx. 
$$
We shall recover some of theses estimates in Lemma \ref{pp} below. \qed
\end{remark}

\begin{remark} (Notations for Section \ref{sec-exp}) In what follows, we will be in the context of Theorem \ref{thm-attractor}. 
To simplify notations, without loss of generality, we take $\ep=1$ and $S(t)$ instead $S_1(t)$. 
Given two initial data $(u_0,u_1)$ and $(v_0,v_1)$ in a bounded set $B \subset \mathcal{H}$, 
we shall use notation   
$$
(u(t),\partial_t u(t)) = S(t)(u_0,u_1) \;\; \mbox{and} \;\; (v(t),\partial_t v(t)) = S(t)(v_0,v_1).
$$ 
Putting $w=u-v$ we write 
$$
E(t)=\frac{1}{2}\|(w(t),\partial_t w(t))\|^2_{\mathcal{H}}.
$$
From Remark \ref{rem-absorb}, $S(t)$ possesses a bounded absorbing set $\mathcal{B}$, which can be assumed closed and forward invariant. 
Moreover, $C_{T}$, $C_{BT}$ and $C_{\mathcal{B}T}$ will denote several positive constants with obvious meaning. \qed 
\end{remark}

\begin{lemma} $($\cite[Proposition 6]{clt}$)$ \label{pp} 
Given $T>0$ sufficiently large, there exist constants $\sigma \in (0,1)$ and $C_T>0$ such that
\begin{equation} \label{key-estimate}
E(T) + C_{T}\int_{0}^{T}E(t)dt\le \sigma E(0) + C_{BT}\sup_{t \in [0,T]}\|w(t)\|^2_2 + C_{T} \Upsilon , 
\end{equation}
where 
\begin{align} \label{upsilon}
\Upsilon = 
& \int_{0}^{T}\left| \int_{0}^{t}\int_{\Omega}(f(u)-f(v)) \partial_t w \, dxd\tau\right|dt 
+ \int_{0}^{T}\left| \int_{s}^{T}\int_{\Omega}(f(u)-f(v)) \partial_t w \, dxd\tau\right|ds \nonumber \\ 
& + \left|\int_{0}^{T}\int_{\Omega} (f(u)-f(v)) \partial_t w \, dxd\tau\right|. 
\end{align}  
\end{lemma}

Lemma \ref{pp} allows us to prove the following key estimate. 

\begin{lemma} \label{propprop}
Let $\mathcal{B}$ be a closed positively invariant bounded absorbing set of the system $(\mathcal{H},S(t))$. Then, there exist positive constants $C_{\mathcal{B}T}$ and   
$\sigma \in (0,1)$ such that
\begin{equation} \label{clave}
E(s+T)\le \sigma E(s)+C_{\mathcal{B}T} \sup_{\tau \in [0,T]}\|w(s+\tau)\|^2_2, \;\; s \ge 0.
\end{equation}
\end{lemma}

\begin{proof} 
Firstly, we see that given $\nu>0$, 
\begin{align} \label{key-estimate2}
\left|\int_{s}^{t} \int_{\Omega}(f(u)-f(v))\partial_t w \,dx d\tau \right| \le 
& \; \nu (E(s)+E(t)) + C_{\mathcal{B}\nu} \sup_{\tau \in [s,t]} \| w(\tau) \|_2^2 \nonumber \\ 
& + \nu \int_{s}^{t}\|\nabla w\|_2^2 \, d\tau, \;\; 0 \le s \le t,
\end{align}
for some constant $C_{\mathcal{B}\nu} > 0$. This estimate is analogous to the one in \cite[Proposition 8]{clt}.  
The only difference is that trajectories (not necessarily complete) 
are bounded in $\mathcal{B}$ instead within attractor $\mathcal{A}$.

Then, in view of definition of $\Upsilon$ in \eqref{upsilon}, we have from (\ref{key-estimate2}), 
$$
\Upsilon \le \nu C_T(E(0)+E(T))+C_{\mathcal{B} T \nu}\sup_{\tau \in [0,T]}\|w(\tau)\|_2^2 + \nu C_T \int_{0}^{T}E(t)dt,
$$
for certain $C_{\mathcal{B} T \nu } >0$. Then, taking $\nu>0$ small enough, the estimate (\ref{key-estimate}) becomes 
$$
E(T)\le \sigma^* E(0) + C_{\mathcal{B}T} \sup_{t \in [0,T]} \| w(t) \|^2_2,
$$
with $\sigma^* \in (0,1)$ and some constant $C_{\mathcal{B}T} >0$. Since $\mathcal{B}$ is 
closed forward invariant, we obtain for any $s\ge 0$,  
$$
E(s+T)\le \sigma^* E(s) + C_{\mathcal{B}T}\sup_{t \in [0,T]}\|w(s+t)\|^2_2,
$$
which shows \eqref{clave}. 
\end{proof}

Now, we are in a position to consider Theorem \ref{thm-lasi} in our context. Given $T>0$, large enough, we define 
\begin{equation} \label{V}
\mathcal{V} = \mathcal{H} \times W(0,T),
\end{equation}
where
$$
W(0,T) = \{w \in H^1(0,T;L^2(\Omega)) \, | \, \| w \|_{W} = \| (w,\partial_t w) \|_{L^2(0,T;\mathcal{H})} <\infty \}.
$$
It is clear that
\begin{equation} \label{norm}
\|(w(0),\partial_t w(0))\|^2_{\mathcal{V}} = \|(w(0),\partial_t w(0))\|^2_{\mathcal{H}} + \int_{0}^{T} \|(w(t),\partial_t w(t))\|^2_{\mathcal{H}} \, dt.
\end{equation}
Additionally we define 
$$
\mathcal{B}_T = \{\left(z_0,S(t)z_0\right) \, | \, z_0 \in \mathcal{B}, \; t \in [0,T] \} \subset \mathcal{V},
$$
where $\mathcal{B}$ is a closed invariant bounded absorbing set as in Lemma \ref{propprop}, 
and define $V:\mathcal{B}_T \to \mathcal{V}$ by setting 
$$
V(z_0,S(t)z_0) = (S(T)z_0 , S(T+t)z_0).
$$

\begin{lemma} \label{propkey}
Under above definitions we have: 
\begin{enumerate}
\item Let $B_T$ be a bounded subset of $\mathcal{B}_T$. Then there exists a constant $C_{BT}>0$ such that  
\begin{equation} \label{key1}
\|VU_1-VU_2\|_{\mathcal{V}}\le C_{BT} \|U_1-U_2\|_{\mathcal{V}},
\end{equation}
for any $U_1, U_2 \in B_T$. \smallskip 

\item There exist constants $K_T>0$ and $\sigma \in (0,1)$ such that, 
for any 
$$
Z_1=(u(0),\partial_t u(0),u(t)) \;\; \mbox{and} \;\; Z_2=(v(0),\partial_t v(0),v(t)) \;\; \mbox{in} \;\; \mathcal{B}_T, 
$$ 
\begin{align} \label{key-ine}
\| VZ_1 - VZ_2\|_{\mathcal{V}} \le 
& \; \sigma \|Z_1-Z_2\|_{\mathcal{V}} \nonumber \\ 
& + K_T\left(\sup_{s \in [0,T]}\|w(s)\|_{L^2(\Omega)}+\sup_{s \in [0,T]}\|w(T+s)\|_{L^2(\Omega)} \right), 
\end{align}
where $w=u-v$. \smallskip 
\item The mapping $V$ possesses a positively invariant compact set $\mathcal{A}_{T} \subset \mathcal{B}_T$, of finite fractal dimension, and
\begin{equation} \label{final}
{\rm dist}_{\mathcal{V}}(V^k\mathcal{B}_T,\mathcal{A}_{T})\le rq^k, \;\; k \in \mathbb{N},
\end{equation}
for some $r>0$ and $q \in (0,1)$.
\end{enumerate}
\end{lemma}

\begin{proof} 
The Lipschitz condition \eqref{key1} follows from (\ref{S1}) and (\ref{V}). To prove \eqref{key-ine}, 
we integrate (\ref{clave}) over $s \in [0,T]$ and add it to (\ref{clave}), obtaining 
$$
E(T)+\int_{T}^{2T}E(T)\le \sigma \left(E(0)+\int_{0}^{T}E(s)ds \right) +2T C_{\mathcal{B}T} \sup_{\tau \in [0,2T]}\|w(\tau)\|_2^2 .
$$
From definition of norm (\ref{norm}) we see \eqref{key-ine}. 
	
To prove the last statement we apply Theorem \ref{thm-lasi} with $M=\mathcal{B}_T$, $H=\mathcal{V}$. 
Indeed, we note that $\sup_{s \in [0,T]}\|\cdot\|_{2}$ defines a compact seminorm in $\mathcal{V}$. 
In addition, since $\mathcal{B}$ is closed forward invariant, we have $V\mathcal{B}_T\subseteq\mathcal{B}_T$. 
Then Theorem \ref{thm-lasi} grants \eqref{final}. 
\end{proof}

For the next lemma we define the ultra-weak phase space 
$$
\mathcal{H}_{-1} = L^2(\Omega) \times H^{-1}(\Omega). 
$$

\begin{lemma} \label{key-lemma-exp}
Let $\mathcal{B}$ be a closed forward invariant bounded absorbing set. 
Then there exist a constant $C_{\mathcal{B}T}>0$ such that
\begin{equation} \label{key-holder-estimative}
\|S(t_1)z-S(t_2)z\|_{\mathcal{H}_{-1}} \le C_{\mathcal{B}T}|t_1-t_2|, \;\; t_1,t_2 \in [0,T], \ z \in \mathcal{B}.
\end{equation}
\end{lemma}

\begin{proof} 
The Cauchy problem \eqref{P-U} implies that  	
\begin{align*}
\|\partial_t U(t)\|_{\mathcal{H}_{-1}} 
& \le \|\mathbb{A}U(t)\|_{{\mathcal{H}_{-1}}} + \|{\mathbb{F}}U(t)\|_{{\mathcal{H}_{-1}}}+\|\mathbb{H}\|_{{\mathcal{H}_{-1}}} \\
& \le C \left(  \|\mathbb{A}U(t)\|_{{\mathcal{H}}}+\|{\mathbb{F}}U(t)\|_{{\mathcal{H}}}+\|\mathbb{H}\|_{{\mathcal{H}}}\right). 
\end{align*}
Now, since $\mathbb{F}$ is locally Lipschitz and $\mathcal{B}$ is bounded, the estimate \eqref{S1} implies that 
$$
\|\mathbb{A}U(t)\|_{\mathcal{H}} + \|\mathbb{F}U(t)\|_{\mathcal{H}} + \|\mathbb{H}\|_{\mathcal{H}} \le C_{\mathcal{B}T},
$$
for some constant $C_{\mathcal{B}T}>0$. Since $U(t) = S(t)z$, we see for $0\le t_1\le t_2\le T$, 
	$$
	\|S(t_2)z - S(t_1)z \|\le \int_{t_1}^{t_2}\|\partial_t U(s)\|_{\mathcal{H}_{-1}}ds\le C_{\mathcal{B}T}|t_2-t_1|,
	$$
which implies \eqref{key-holder-estimative}.
\end{proof}

\noindent{\bf Proof of Theorem \ref{thm-exp}.} For $T>0$ large enough, we obtain from \eqref{final}, 
$$
{\rm dist}_{\mathcal{V}}(V^k \mathcal{B}_T,\mathcal{A}_T)\le rq^k, \;\; k \in \mathbb{N}, 
$$
for some $r>0$ and $q \in (0,1)$. In particular, 
\begin{equation} \label{SKT}
	{\rm dist}_{\mathcal{H}}(S(kT)\mathcal{B},\mathbb{P}\mathcal{A}_T)\le rq^k, \;\; k \in \mathbb{N}, 
\end{equation}
where $\mathbb{P}:\mathcal{V}\to \mathcal{H}$ is the projection of $\mathcal{A}_T$ over the first component, that is,
$$
\mathbb{P}\mathcal{A}_T = \{z_0 \in \mathcal{B} \ | \ (z_0,S(t)z_0) \in \mathcal{A}_T, \ t \in [0,T] \}.
$$
It is clear that $\mathbb{P}\mathcal{A}_T$ is a compact set in $\mathcal{H}$ and $S(T)\mathbb{P}\mathcal{A}_T\subseteq \mathbb{P}\mathcal{A}_T$. Moreover, 
\begin{equation} \label{dim-AT}
{\rm dim}^{\mathcal{H}}_f \mathbb{P}\mathcal{A}_T \le {\rm dim}^{\mathcal{V}}_f \mathcal{A}_T <\infty,
\end{equation}
We define the compact set in $\mathcal{H}$ (candidate to exponential attractor)
$$
\mathcal{A}^{\rm exp}:=\bigcup_{t \in [0,T]} S(t)\mathbb{P}\mathcal{A}_T.
$$
Then, by construction $S(t)\mathcal{A}^{\rm exp} \subseteq \mathcal{A}^{\rm exp}$. 
In addition, from (\ref{S1}) and (\ref{SKT}) we see that 
$$
{\rm dist}_{\mathcal{H}}(S(t)\mathcal{B},\mathcal{A}^{\rm exp})\le Ce^{-\gamma t}, \ \ t\ge 0,
$$
for some $C,\gamma>0$. It remains to show that $\mathcal{A}^{\rm exp}$ has finite fractal dimension in some space 
containing $\mathcal{H}$. We shall use Lemma \ref{key-lemma-exp}. 

Let us define a mapping 
$$
\mathcal{F} : \mathbb{R} \times \mathcal{H} \to \mathcal{H}_{-1}  
\;\; \mbox{such that} \;\; \mathcal{F}(t,z) = S(t)z, \;\; t\ge 0. 
$$
Then we have  
$$
\mathcal{A}^{\rm exp} = \mathcal{F}([0,T]\times \mathbb{P} \mathcal{A}_{T}).  
$$
We claim that $\mathcal{F}$ is Lipschitz restricted to $[0,T]\times \mathbb{P} \mathcal{A}_{T}$. Indeed, taking into account that $\mathbb{P} \mathcal{A}_T \subset \mathcal{B}$, estimates \eqref{S1} and \eqref{key-holder-estimative} imply that, 
\begin{align*}
\| \mathcal{F}(t_1,z_1) - \mathcal{F}(t_2,z_2) \|_{\mathcal{H}_{-1}}  
& \le \| S(t_1)z_1 - S(t_1)z_2 \|_{\mathcal{H}_{-1}} + \| S(t_1)z_2 - S(t_2)z_2 \|_{\mathcal{H}_{-1}} \\ 
& \le C_{\mathcal{B}T} \| z_1 - z_2 \|_{\mathcal{H}} + C_{\mathcal{B}T} | t_1 - t_2 | \\ 
& \le L \| (t_1,z_1) - (t_2,z_2) \|_{\mathbb{R} \times \mathcal{H}}, 	
\end{align*}
for some $L>0$. This proves the claim. Now, since Lipschitz mapping does not increase fractal dimension (e.g. \cite[Proposition C.1]{efnt}), we conclude that 
\begin{align*}  
{\rm dim}_f^{\mathcal{H}_{-1}} \mathcal{A}^{\rm exp}
& \le {\rm dim}_f^{\mathbb{R}\times \mathcal{H}} ([0,T]\times \mathbb{P} \mathcal{A}_{T}) \smallskip \\ 
& \le 1 + {\rm dim}^{\mathcal{H}}_f \mathbb{P}\mathcal{A}_T. 
\end{align*}
Then from \eqref{dim-AT} it follows that ${\rm dim}_f^{\mathcal{H}_{-1}} \mathcal{A}^{\rm exp} < \infty$. 
Thus, $\mathcal{A}^{\rm exp}$ is a generalized exponential attractor for $(\mathcal{H},S(t))$ with finite fractal 
dimension in $\widetilde{\mathcal{H}} = \mathcal{H}_{-1}$. This completes the proof of Theorem \ref{thm-exp}. \qed 

\medskip

\section*{Acknowledgements} This paper was done while the second author was visiting the Department of Mathematics of University of Chile, whose kind hospitality is gratefully acknowledged. 
He also thanks professors Gonzalo Robledo Rodr\'{\i}gues and \'{A}lvaro Casta\~{n}eda Gonzales for arranging funding support from 
FONDECYT, REGULAR grant 1170968. The first author was partially supported by CNPq grant 312529/2018-0.


\end{document}